\documentclass[a4paper,10pt,reqno]{amsart}
\usepackage{ascmac,amsmath,amsthm,amssymb,amscd,mathrsfs,color,setspace,stmaryrd}
\usepackage[all]{xy}
\usepackage{enumerate}
\usepackage[dvips]{graphicx}
\usepackage{latexsym}
\numberwithin{equation}{section}
\newtheorem{thm}{Theorem}[section]
\newtheorem{prop}[thm]{Proposition}
\newtheorem{cor}[thm]{Corollary}
\newtheorem{lem}[thm]{Lemma}
\theoremstyle{definition}
\newtheorem{df}[thm]{Definition}

\newtheorem{ex}[thm]{Example}

\def\gldim{\mathop{{\mathrm{gl.dim}}}\nolimits}

\def\Im{\mathop{\mathrm{Im}}\nolimits}

\def\Hom{\mathop{\mathrm{Hom}}\nolimits}
\def\End{\mathop{\mathrm{End}}\nolimits}
\def\Ext{\mathop{\mathrm{Ext}}\nolimits}
\def\Tor{\mathop{\mathrm{Tor}}\nolimits}

\def\Lten{\mathop{\otimes^{\mathbb L}_\Lambda}\nolimits}
\def\RHom{\mathop{\mathbb R\mathrm{Hom}}\nolimits}

\def\proj{\mathop{\mathrm{proj}}\nolimits}
\def\Mod{\mathop{\mathrm{Mod}}\nolimits}
\def\mod{\mathop{\mathrm{mod}}\nolimits}

\def\add{\mathop{\mathrm{add}}\nolimits}
\def\Gr{\mathop{\mathrm{Mod}^{\mathbb{Z}}}\nolimits}
\def\gr{\mathop{\mathrm{mod}^{\mathbb{Z}}}\nolimits}

\def\Cen{\mathop{\mathrm{Z}}\nolimits}
\newcommand{\bdcat}[2]{\mathcal{D}^{\mathrm{b}}_{#1}(#2)}

\def\op{\mathop{\mathrm{op}}}
\newcommand{\La}{\Lambda}

\title{Higher APR tilting preserves $n$-representation infiniteness}
\author[Mizuno]{Yuya Mizuno}
\address{Y. Mizuno: Graduate School of Mathematics\\ 
 Nagoya University, Chikusa-ku Nagoya 464-8602 Japan}
\email{yuya.mizuno@math.nagoya-u.ac.jp}

\author[Yamaura]{Kota Yamaura}
\address{K. Yamaura: Department of Research Interdisciplinary Graduate School of Medicine and Engineering \\ 
University of Yamanashi, 4-4-37 Takeda Kofu Yamanashi 400-8510 Japan}
\email{kyamaura@yamanashi.ac.jp}

\begin{document}

\maketitle

\begin{abstract}
We show that $m$-APR tilting preserves $n$-representation infiniteness for $1\leq m\leq n$. 
Moreover, we show that these tilting modules lift to tilting modules for the corresponding higher preprojective algebras, which is $(n+1)$-CY algebras. 
We also study the interplay of the two kinds of tilting modules.
\end{abstract}


\section{Introduction}

The notion of Calabi-Yau (CY) triangulated categories has been used in many branch of mathematics. 
In the representation theory of algebras, this already appeared in the work of Auslander on Cohen-Macaulay representations of Gorenstein orders \cite{Au}. 
During the last few years,  the cluster tilting theory in 2-CY categories has been rapidly developed, which provides a categorification of cluster algebras (see an expository paper \cite{Ke3}). 
More generally, it has been recently recognized that $n$-CY categories have several nice properties together with many applications, for instance  \cite{IYo,Ke2,KR1,KR2,IYa}. 
Therefore, from the viewpoint of homological algebras, it is quite natural to study \emph{Calabi-Yau algebras}, which provide CY categories as (subcategories of) derived categories, and, indeed, their constructions have been extensively investigated, for example, by \cite{AIR,BSW,Bro,Boc,G,TV,Vdb}.

One of the aim of this paper is to give a new method to construct a family of CY algebras via tilting modules. 
For this purpose, we apply higher dimensional Auslander-Reiten theory, which was introduced in \cite{I1,I2,I3} and has been developed, for example, in \cite{HI1,HI2,IO1,IO2,HIO}. 
Recently it has turned out that $(n+1)$-CY algebras have a deep connection with the notion of \emph{$n$-representation infinite algebras} and we have the following one-to-one correspondence (see Definition \ref{defCY} and Theorem \ref{corr_RI_CY})
\[
\begin{xy}
(-40,0)*{\left\{ \begin{array}{c} \mbox{$n$-representation} \\ \mbox{infinite algebras}  \end{array} \right\}}="A",
(40,0)*{\left\{ \begin{array}{c}\mbox{bimodule $(n+1)$-Calabi-Yau} \\ \mbox{algebras of G. P. $1$} \end{array} \right\}}="B",,

\ar "A";"B"^{{\tiny \begin{array}{c}  \mbox{$(n+1)$-preprojective} \hspace{10mm} \\  \mbox{algebras}\hspace{10mm} \end{array}}}
\end{xy}
\]

In this paper, through the connection, 
we study tilting modules for $n$-representation infinite algebras and, at the same time, for CY algebras. 
The key tool is the notion of \emph{m-APR tilting} modules (Definition \ref{df_APR}), which is a generalization of classical reflection functors. 
In \cite{IO1,HIO}, it is shown that $n$-APR tilting modules preserve $n$-representation finiteness  ($n$-representation infiniteness) for $n\geq 1$. 
Our first result is that, even for $m$ with $1\leq m\leq n$, $m$-APR tilting modules preserve $n$-representation infiniteness.  
By this fact, we obtain a large family of  $n$-representation infinite algebras. 
Our second result is that these modules lift to tilting modules over the corresponding $(n+1)$-preprojective algebras. 
Moreover, we show that the $(n+1)$-preprojective algebra of an $m$-APR tilted algebra is isomorphic to the endomorphism algebra  of the corresponding tilting module induced by the  $m$-APR tilting module.

\[
\begin{xy}
(-45,0)*{\left\{ \begin{array}{c} \mbox{$n$-representation} \\ \mbox{infinite algebras}  \end{array} \right\}}="A",
(45,0)*{\left\{ \begin{array}{c} \mbox{$n$-representation} \\ \mbox{infinite algebras}  \end{array} \right\}}="B",
(-45,-25)*{\left\{ \begin{array}{c} \mbox{bimodule $(n+1)$-Calabi-Yau} \\ \mbox{algebras of G. P. $1$}  \end{array} \right\}}="C",
(45,-25)*{\left\{ \begin{array}{c} \mbox{bimodule $(n+1)$-Calabi-Yau} \\ \mbox{algebras of G. P. $1$}  \end{array} \right\}}="D",

\ar "A";"B"^{{\tiny \begin{array}{c} \mbox{higher APR tilting}  \\ \mbox{(Theorem \ref{main_thm1})} \end{array}}}
\ar "A";"C"_{{\tiny \begin{array}{c} \mbox{$(n+1)$-preprojective} \\ \mbox{algebras} \end{array}}}
\ar "B";"D"^{{\tiny \begin{array}{c} \mbox{$(n+1)$-preprojective} \\ \mbox{algebras} \end{array}}}
\ar "C";"D"_{{\tiny \begin{array}{c} \mbox{corresponding tilting} \\ \mbox{(Theorem \ref{main_thm2})}  \end{array}}}
\end{xy}
\]

This also implies that we obtain a family of $(n+1)$-CY algebras, which are derived equivalent to each other. As an application, we prove that $n$-representation tameness is also preserved under $m$-APR tilting (Corollary \ref{pres_rep_tame}).

\bigskip

\noindent
\textbf{Notations.}
Let $K$ be an algebraically closed field. 
We denote by $D:=\Hom_K(-,K)$ the $K$-dual. 
An algebra means a $K$-algebra which is indecomposable as a ring. 
For an algebra $\Lambda$, we denote by $\Mod\Lambda$ the category of right $\Lambda$-modules and 
by $\mod\Lambda$ the category of finitely generated $\Lambda$-modules. If $\La$ is $\mathbb{Z}$-graded, 
we denote by $\Gr\Lambda$ the category of  $\mathbb{Z}$-graded ${\Lambda}$-modules and by 
$\gr {\Lambda}$ the category of finitely generated $\mathbb{Z}$-graded ${\Lambda}$-modules.

\bigskip
\noindent
\textbf{Acknowledgements.} 
The authors thank Osamu Iyama for his support and kind advice.
They are grateful to Hiroyuki Minamoto for useful discussion. 
Part of this work was done while they visited Trondheim. They thank all the people at NTNU for their hospitality.
The first author thanks the Institute Mittag-Leffler for the support and warm hospitality during the preparation of this paper. 
The first author is supported by JSPS Grant-in-Aid for Scientific Research 26800009
and the second author is supported by JSPS Grant-in-Aid for Scientific Research 26800007.

\section{Preliminaries}

In this section, we recall basic notions and notations appearing in this paper.

\subsection{$n$-representation infinite algebras}

Let $\Lambda$ be a finite dimensional algebra of global dimension at most $n$.  
We let $\bdcat{}{\Lambda}:=\mathcal{D}^{\mathrm{b}}(\mod\La)$ and denote the Nakayama functor by 
 \[
\nu:=- \Lten D\Lambda\simeq D\RHom_\La(-,\La): \bdcat{}{\Lambda} \longrightarrow \bdcat{}{\Lambda}.
\]
Then $\nu$ gives a Serre functor,  i.e. there exists a functorial isomorphism 
\[
\Hom_{\bdcat{}{\Lambda}}(X,Y) \simeq D\Hom_{\bdcat{}{\Lambda}}(Y,\nu X)
\]
for any $X, Y \in\mathcal{D}^{\mathrm{b}}(\Lambda)$. 
A quasi-inverse of $\nu$ is given by
\[
\nu^-:=\RHom_\La(D\La,-)\simeq - \Lten \RHom_{\Lambda}(D\Lambda,\Lambda):\bdcat{}{\Lambda} \longrightarrow \bdcat{}{\Lambda}.
\]

We let 
\[
\nu_n:=\nu \circ [-n] \ \textnormal{and}\ \nu_n^-:=\nu^-\circ [n].
\]

By uniqueness of the Serre functor, we have the following property. 

\begin{lem}\label{Serre-com}
Let $\Gamma$ be a finite dimensional algebra.  If there is a triangle-equivalence $F:\bdcat{}{\Lambda} \to \bdcat{}{\Gamma}$, 
 then the following diagram commutes up to isomorphisms of functors.  
\[
\xymatrix{
\bdcat{}{\Lambda} \ar[r]^{\nu_n} \ar[d]_{F} & \bdcat{}{\Lambda}  \ar[d]^{F} \\
\bdcat{}{\Gamma} \ar[r]_{\nu_n} & \bdcat{}{\Gamma}. 
}
\]

\end{lem}

Next we define the \emph{$n$-Auslander-Reiten translations} 
\[
\tau_n=D\Ext_\La^n(-,\La):\mod\Lambda \to \mod\Lambda \ \textnormal{and}\ 
\tau_n^-=\Ext_\La^n(D\La,-)\simeq-\otimes_\La\Ext_\La^n(D\La,\La) :\mod\Lambda \to \mod\Lambda.
\]
These functors play an important role in higher dimensional Auslander-Reiten theory \cite{I1,I2,I3}. 
Note that, since $\gldim\La\leq n$, we have  
$
\tau_n=\mathrm{H}^0(\nu_n-)$ and $\tau_n^-=\mathrm{H}^0(\nu_n^--).
$

We consider the following full subcategories
\[
\mathscr{N}_\Lambda^- := \mathrm{add}\{ X \in \mathrm{mod}\Lambda \ | \ \mbox{$\nu^{-i}_n(X) \in \mathrm{mod}\Lambda$ for any $i \geq 0$} \},\]
\[\mathscr{N}_\Lambda^+ := \mathrm{add}\{ X \in \mathrm{mod}\Lambda \ | \ \mbox{$\nu^{i}_n(X) \in \mathrm{mod}\Lambda$ for any $i \geq 0$} \}\]
of $\mathrm{mod}\Lambda$. 
We often use the fact that $X \in \mathscr{N}_\Lambda^-$ if and only if $\mathrm{H}^j(\nu^{-i}_n(X))=0$ for any $j\neq 0$ and $i\geq0$. 

Now we recall the definition of $n$-representation infinite algebras as follows.

\begin{df} \cite{HIO}
A finite dimensional algebra $\Lambda$ is called \emph{$n$-representation infinite} if it satisfies $\gldim \Lambda \leq n$ and $\Lambda \in \mathscr{N}_\Lambda^-$, which is equivalent to say that $\gldim \Lambda \leq n$ and $D\Lambda \in \mathscr{N}^+_\Lambda$.
\end{df}

We remark that 1-representation infinite algebras are nothing but  
hereditary representation infinite algebras. 
Thus $n$-representation infinite algebras can be regarded as a generalization of them.

Moreover, by the definition, we have the following property, which will be used later.

\begin{lem}\label{lem0}
Let $\Lambda$ be an $n$-representation infinite algebra.  
Let $X \in \mathscr{N}_\Lambda^-$ and $Y \in \mathscr{N}^+_\Lambda$.
Then there exist isomorphisms $\nu^{-i}_n(X) \simeq\tau^{-i}_n(X) $ and $\nu_n^{i}(Y) \simeq\tau_n^{i}(Y) $ for any $i\geq0$.
\end{lem}

\subsection{$m$-APR tilting modules}

Let $\Lambda$ be an algebra. 
A $\Lambda$-module $T$ is called \emph{tilting} if it satisfies the following conditions.
\begin{enumerate}
\def\labelenumi{(\theenumi)} 
\item[(T1)] There exists an exact sequence
\[
0 \to P_m \to \cdots\cdots P_1 \to P_0 \to T \to 0
\]
where each $P_i$ is a finitely generated projective $\Lambda$-module. 
\item[(T2)] $\Ext_{\Lambda}^i(T,T)=0$ for any $i>0$.
\item[(T3)] 
There exists an exact sequence 
\[
0 \to \Lambda \to T_{0} \to T_1 \to \cdots\cdot \to T_m \to 0
\]
where each $T_i$ belongs to $\add T$. 
\end{enumerate}
In this case, there exists a triangle-equivalence between $\bdcat{}{\Lambda}$
and $\bdcat{}{\End_{\Lambda}(T)}$ \cite{H,Ric}.

The following tilting modules play a central role in this paper.

\begin{df} \label{df_APR}
Let $\Lambda$ be a finite dimensional algebra of global dimension at most $n$. 
We assume that there is a simple projective $\Lambda$-module $S$ satisfying $\Ext^i_{\Lambda}(D\Lambda,S)=0$ for any $0 \leq i < n$. Take a direct sum  decomposition $\Lambda=S \oplus Q$ as a $\Lambda$-module. 
In \cite[Proposition 3.2]{IO1} (and its proof), it was shown that there exists a minimal projective resolution 
\begin{eqnarray*}
0 \to S \xrightarrow{a_0} P_1\xrightarrow{a_1} \cdots \cdots \xrightarrow{a_{n-1}} P_{n}\xrightarrow{a_n} \tau_n^-(S) \to 0 
\end{eqnarray*}
of $\tau^-_n(S)$ such that each $P_i$ belongs to $\add Q$.
Let $K_m:=\Im a_m$ for $0 \leq m \leq n$. 
Note that $K_0=S$ and $K_n=\tau_n^-(S)$.
Then it was shown that  $Q \oplus K_m$ is a tilting module with projective dimension $m$. Following \cite{IO1}, we call it the \emph{$m$-APR} (=Auslander-Platzeck-Reiten) tilting module with respect to $S$.  
\end{df}

If $\Lambda$ is an $n$-representation infinite algebra with a simple projective module $S$, then  the above condition $\Ext^i_{\Lambda}(D\Lambda,S)=0$  ($0 \leq i <n$) is automatically satisfied.
Thus any simple projective module gives an $m$-APR tilting module for $n$-representation infinite algebras.

\subsection{$(n+1)$-preprojective algebras}
Next we recall the definition of $(n+1)$-preprojective algebras and their properties. 
In the case of $n=1$, the algebras coincide with the (classical) preprojective algebras \cite{GP,BGL,Rin}.

\begin{df}\label{df_pp}\cite{IO2}
Let $\La$ be a finite dimensional algebra. 
The \emph{$(n+1)$-preprojective algebra} $\widehat{\Lambda}$ for  $\Lambda$ is a tensor algebra
\[
\widehat{\Lambda} := T_{\Lambda}(\Ext^n_{\Lambda}(D\Lambda,\Lambda))
\]
of $\Lambda^{\op}\otimes_K \Lambda$-module $\Ext^n_{\Lambda}(D\Lambda,\Lambda)$.
This algebra can be regarded as a positively graded algebra by
\[
\widehat{\Lambda}_i= \Ext^n_{\Lambda}(D\Lambda,\Lambda)^{\otimes_{\Lambda}^i}
= \overset{i}{\overbrace{\Ext^n_{\Lambda}(D\Lambda,\Lambda) \otimes_{\Lambda} \cdots \cdots \otimes_{\Lambda}\Ext^n_{\Lambda}(D\Lambda,\Lambda)}}.
\]
\end{df}

We remark that the $(n+1)$-preprojective algebra is the 0-th homology 
of Keller's \emph{derived $(n+1)$-preprojective DG algebra} \cite{Ke2}.

Moreover, we recall the following definition, which is a graded analog of 
Ginzburg's Calabi-Yau algebras \cite{G}.

\begin{df}\label{defCY}
Let $A=\bigoplus_{i \geq 0}A_i$ be a positively graded algebra such that $\dim_K A_i < \infty$ for any $i \geq 0$.
We denote by $A^e:=A^{\op} \otimes_KA$. 
We call $A$ \emph{bimodule $n$-Calabi-Yau of Gorenstein parameter $1$} if it satisfies the following conditions.
\begin{enumerate}
\def\labelenumi{(\theenumi)} 
\item $A \in \mathcal{K}^{\mathrm b}(\proj^{\mathbb{Z}} A^e)$.
\item $\RHom_{A^e}(A,A^e)[n](-1) \simeq A$ in $\mathcal{D}(\Gr A^e)$.
\end{enumerate}
\end{df}

Then $n$-representation infinite algebras and bimodule CY algebras have a close relationship as follows (see \cite[Theorem 4.35]{HIO}).

\begin{thm}\cite{AIR,Ke2,MM}\label{corr_RI_CY}
There is a one-to-one correspondence between
isomorphism classes of $n$-representation infinite algebras $\Lambda$ and
isomorphism classes of graded bimodule $(n+1)$-CY algebras $A$ of Gorenstein parameter $1$. The correspondence is given by
\[
\Lambda\mapsto \widehat{\Lambda}\ \ {and}\ 
\  A\mapsto A_0.
\]
\end{thm}

We remark that the geometric rolled-up helix algebra on a smooth projective Fano variety is given as the $(n+1)$-preprojective algebra of an $n$-representation infinite algebra (see \cite{BS,TV} and Example \ref{example} (2)).

The following result implies that bimodule $n$-CY algebras provide $n$-CY triangulated categories.

\begin{thm}\label{CY_condition}\cite[Lemma 4.1]{Ke1}\cite[Proposition 3.2.4]{G}
Let $A$ be a bimodule $n$-CY algebra. 
Then there exists a functorial isomorphism 
\[
\Hom_{\mathcal{D}(\Mod A)}(M,N) \simeq D\Hom_{\mathcal{D}(\Mod A)}(N,M[n])
\]
for any $N \in \mathcal{D}(\Mod A)$ whose total homology is finite dimensional and any $M \in \mathcal{D}(\Mod A)$.
\end{thm}


\section{Main result}

Throughout this section, 
let $\Lambda$ be an $n$-representation infinite algebra. Assume that there exists a simple projective $\Lambda$-module $S$ and take a direct sum decomposition $\La=S\oplus Q$ as a $\La$-module. 
As Definition \ref{df_APR}, 
we have a minimal projective resolution 
\begin{eqnarray}
0 \to S \xrightarrow{a_0} P_1\xrightarrow{a_1} \cdots \cdots \xrightarrow{a_{n-1}} P_{n}\xrightarrow{a_n} \tau_n^-(S) \to 0 \label{ast}
\end{eqnarray}
of $\tau^-_n(S)$ such that each $P_i$ belongs to $\add Q$. 
Let $K_i:=\Im a_i$ and we fix $m$ with  $0 \leq m \leq n$. 
Then we denote, respectively, the $m$-APR tilting $\Lambda$-module and the endomorphism algebra by 
\begin{eqnarray}\label{notation}
T:=Q\oplus K_m\ \ \ \ \ \ \  \mbox{and}\ \ \ \ \ \ \  \Gamma:=\End_{\Lambda}(T).
\end{eqnarray}

\subsection{$m$-APR tilting modules preserve $n$-representation infiniteness}

The aim of this subsection is to show the following result, which is a generalization of \cite[Theorem 2.13]{HIO}.

\begin{thm}\label{main_thm1}
Under the above setting, the algebra $\Gamma$ is $n$-representation infinite.
\end{thm}

To prove the above theorem, 
we first give the following lemma.

\begin{lem}\label{lem1}
Let $0 \rightarrow X \rightarrow Y \rightarrow Z \rightarrow 0$ be an exact sequence in $\mathrm{mod}\Lambda$. 
If  $X,Y \in \mathscr{N}_\Lambda^-$ and $\Ext^1_{\Lambda}(Z,\tau^{i}_n (D\Lambda))=0$ holds for any $i \geq 1$,
then we have $Z \in \mathscr{N}_\Lambda^-$.
\end{lem}

\begin{proof}
We show that $\nu_n^{-i}(Z) \in \mathrm{mod}\Lambda$ for any $i \geq 0$ by induction on $i$.
We have nothing to show in the case $i=0$.
Let $i \geq 1$. 
We assume that $\nu_n^{-i+1}(Z) \in \mathrm{mod}\Lambda$ and prove that $\mathrm{H}^{j}(\nu_n^{-i}(Z))=0$ for $j  \neq 0$.

The exact sequence gives rise to the triangle $X \to Y \to Z \to X[1]$ in $\bdcat{}{\Lambda}$. 
By applying the functor $\nu_n^{-i}$ to this triangle, we have the triangle $\nu_n^{-i}(X) \to \nu_n^{-i}(Y) \to \nu_n^{-i}(Z) \to \nu_n^{-i}(X[1])$ in $\bdcat{}{\Lambda}$. 
Then, by taking the homology, we have the following long exact sequence
\begin{eqnarray*}
\cdots\cdots  \rightarrow \mathrm{H}^{j-1}(\nu_n^{-i}(Z))  \rightarrow \mathrm{H}^{j}(\nu_n^{-i}(X)) \rightarrow \mathrm{H}^{j}(\nu_n^{-i}(Y)) 
 \rightarrow \mathrm{H}^{j}(\nu_n^{-i}(Z)) \rightarrow \mathrm{H}^{j+1}(\nu_n^{-i}(X))  \rightarrow \cdots\cdots.
\end{eqnarray*}

By this exact sequence, we  have $\mathrm{H}^{j}(\nu_n^{-i}(Z))=0$ for $j\neq 0,1$ because we have $\mathrm{H}^{j}(\nu_n^{-i}(X))=0=\mathrm{H}^{j}(\nu_n^{-i}(Y))$ for any $j \neq 0$.

Moreover, using \cite[Appendix. Proposition 4.11]{ASS}, we have
\begin{eqnarray*}
\mathrm{H}^{-1}(\nu_n^{-i}(Z)) 
&\simeq& \Tor^{\Lambda}_1(\tau_n^{-i+1}(Z),\Ext^n_{\Lambda}(D\Lambda,\Lambda))
\ \ \ \ \ \ \ \ \ \ \  \ \ \ (\nu_n^{-i+1}(Z) \simeq \tau_n^{-i+1}(Z)) \\
&\simeq& D\Ext^1_{\Lambda}(\tau_n^{-i+1}(Z),D\Ext^n_{\Lambda}(D\Lambda,\Lambda)) \\
&\simeq&  D\Ext^1_{\Lambda}(\tau_n^{-i+1}(Z),\tau_n (D\Lambda))\ \ \ \ \ \ \ \ \ \ \ \ \ \ \ \ \ (D(\tau_n^-\Lambda)  \simeq \tau_n (D\Lambda)) \\
&\simeq&  D\Ext^1_{\Lambda}(Z,\tau^{i}_n (D\Lambda))=0.\ \ \ \ \ \ \ \ \ \ \ \ \ \ \ \ \ \ \ \ \  (\mbox{Lemma } \ref{lem0})
\end{eqnarray*}
Therefore, we have $\mathrm{H}^{j}(\nu_n^{-i}(Z))=0$ for $j  \neq 0$ and hence $\nu_n^{-i}(Z)\in\mod\La$. 
\end{proof}

Then we obtain the following consequence.

\begin{prop}\label{van-lem1}
For any $0\leq \ell \leq n$, we have $K_\ell \in \mathscr{N}_\Lambda^-$. 
In particular, we have $T \in \mathscr{N}_\Lambda^-$.
\end{prop}

\begin{proof}
We prove the assertion by induction on $\ell$.
In the case $\ell=0$, we have $K_0 =S \in \mathscr{N}_\Lambda^-$ since $\Lambda$ is $n$-representation infinite.

Assume that $K_\ell \in \mathscr{N}_\Lambda^-$.
Since $\Lambda$ is $n$-representation infinite, we have $P_{\ell+1} \in \mathscr{N}_\Lambda^-$.
Moreover we have
\begin{eqnarray*}
\Ext_{\Lambda}^1(K_{\ell+1},\tau_n^i(D\Lambda))
&\simeq& \Ext_{\Lambda}^{n-\ell}(\tau_n^-(S),\tau_n^i(D\Lambda))\\
&\simeq& \Ext_{\Lambda}^{n-\ell}(S,\tau_n^{i+1}(D\Lambda)) =0 \ \ \ \ \ (\mbox{Lemma } \ref{lem0})
\end{eqnarray*}
for any $i \geq 1$.
Thus we can apply Lemma \ref{lem1} to the exact sequence $0 \to K_\ell \to P_{\ell+1} \to K_{\ell+1} \to 0$ and hence we have $K_{\ell+1} \in \mathscr{N}_\Lambda^-$.
\end{proof}

Moreover we show the following lemma.

\begin{lem}\label{van-lem2}
We have $\Ext_{\Lambda}^j(T, \tau_n^{-i} (T))=0$ for any $j>0$ and $i\geq0$. 
\end{lem}

\begin{proof}
Since $T$ is a tilting module, we have $\Ext_{\Lambda}^j(T, T)=0$ for any $j>0$. 
Assume that $i\geq1$. Then we have 
\begin{eqnarray*}
\Ext_{\Lambda}^j(T, \tau_n^{-i} (T))
&\cong & \Ext_{\Lambda}^{j}(K_m,\tau^{-i}_n (T))\\
&\cong & \Ext_{\Lambda}^{j+n-m}(\tau^{-}_n(S),\tau^{-i}_n (T))\ \ \ \ \ \ \ \ \ (\tau^-_n(S)=K_n)\\
&\cong & \Ext_{\Lambda}^{j+n-m}(S,\tau^{-i+1}_n(T))=0.\ \ \  \ \ (S\in \mathcal{N}^-_\La,\ \mbox{Lemma }\ref{lem0})
\end{eqnarray*}
Thus we have the assertion.
\end{proof}

Now we are ready to prove Theorem \ref{main_thm1}.

\begin{proof}[Proof of Theorem \ref{main_thm1}]
First we show that $\gldim \Gamma \leq n$. 
We have $\gldim\La \leq n$ by definition. On the other hand, $\Lambda$ and $\Gamma$ are derived equivalent. Therefore,  we have $\gldim \Gamma < \infty$. 
Then we have
\begin{eqnarray*}
\gldim \Gamma &=& \max \{ \ell \ | \ \Ext^{\ell}_{\Gamma}(D\Gamma,\Gamma) \neq 0 \} \\
&=& \max \{ \ell \ | \ \Hom_{\bdcat{}{\Gamma}}(\nu (\Gamma),\Gamma[\ell]) \neq 0 \} \\
&=& \max \{ \ell \ | \  \Hom_{\bdcat{}{\Lambda}}(\nu (T),T[\ell]) \neq 0 \}.
\end{eqnarray*}
Moreover, we have
\begin{eqnarray*}
\Hom_{\bdcat{}{\Lambda}}(\nu (T),T[\ell])
&\simeq& \Hom_{\bdcat{}{\Lambda}}(T,\nu^{-1}(T)[\ell]) \\
&\simeq& \Hom_{\bdcat{}{\Lambda}}(T,\nu^{-1}_n(T)[\ell-n]) \\
&\simeq&   \Ext^{\ell-n}_{\Lambda}(T,\nu^{-1}_n(T))\\
&=&   0. \ \ \ \ \ \ \ \ \ \ \ \ \ \ \ \ \ \ \ \ \ \ \ \ \  (\mbox{Lemma \ref{van-lem2}})
\end{eqnarray*}
for any $\ell>n$. Thus we have $\gldim \Gamma \leq n$.

Furthermore, for any $i\geq0$, we obtain 
\begin{eqnarray*}
\nu^{-i}_n (\Gamma)
&\simeq& \nu^{-i}_n \Hom_{\Lambda}(T,T)\\
&\simeq& \nu^{-i}_n \RHom_{\Lambda}(T,T) \ \ \ \ \ \ \ \ \ (\mbox{Lemma }  \ref{van-lem2}) \\
&\simeq& \RHom_{\Lambda}(T,\nu^{-i}_n(T)) \ \ \ \ \ \ \ (\mbox{Lemma }  \ref{Serre-com})  \\
&\simeq& \Hom_{\Lambda}(T,\tau^{-i}_n(T)).\ \ \ \ \ \ \ \  (\mbox{Lemmas } \ref{lem0},\ref{van-lem2}, \mbox{Proposition }  \ref{van-lem1})
\end{eqnarray*}

Thus we have $\Gamma \in \mathscr{N}_\Gamma^-$. 
Consequently $\Gamma$ is $n$-representation infinite.
\end{proof}

\subsection{$m$-APR tilting modules induce tilting modules over higher preprojective algebras}

In this subsection, 
we show that $m$-APR tilting modules over $n$-representation infinite algebras lift to tilting modules over the corresponding $(n+1)$-preprojective algebras. 
We also study a relationship of the endomorphism algebras of these two tilting modules.

Let $\La$ be an $n$-representation infinite algebra and we keep the notation given in the beginning of this section. 
Let $\widehat{\Lambda}:=\bigoplus_{i \geq 0}\widehat{\Lambda}_i$ and $\mathcal{D}(\widehat{\Lambda}):=\mathcal{D}(\Mod\widehat{\Lambda})$. 
For a $\mathbb{Z}$-graded $\widehat{\Lambda}$-module $X$, we write $X_{\ell}$ the degree $\ell$-th part of $X$. 
For a $\mathbb{Z}$-graded finitely generated $\widehat{\Lambda}$-module $X$, 
the algebra  $\End_{\widehat{\Lambda}}(X)$ can be regarded 
as a $\mathbb{Z}$-graded algebra by $\End_{\widehat{\Lambda}}(X)_i=\Hom_{\widehat{\Lambda}}(X,X(i))_0$, where $(i)$ is a graded shift functor and 
$\Hom_{\widehat{\Lambda}}(X,X)_0:=\{ f \in \Hom_{\widehat{\Lambda}}(X,X) \ | \ \mbox{$f(X_i) \subset X_i$ for any $i$} \}$.

Moreover, the algebra $
\Lambda^{\op} \otimes_K \widehat{\Lambda}$ can be regarded as a $\mathbb{Z}$-graded algebra by $(\Lambda^{\op} \otimes_K \widehat{\Lambda})_i  :=\Lambda^{\op} \otimes_K (\widehat{\Lambda})_i$.
Thus we regard $\widehat{\Lambda}$ as a $\mathbb{Z}$-graded $(\Lambda^{\op} \otimes_K \widehat{\Lambda})$-module and we have a functor
\[
\widehat{( \ \ )}:= -\otimes_{\Lambda}\widehat{\Lambda}: \mod\Lambda \longrightarrow \gr \widehat{\Lambda}.
\]
Note that we have 
\[
\widehat{X}_i = \begin{cases} 0 & (i <0) \\ \tau^{-i}_n(X) & ( i \geq 0) \end{cases}
\]
for any $X\in \mod\Lambda$.
Then the aim of  this subsection is to prove the following result.

\begin{thm}\label{main_thm2}
Under the above setting, the following assertions hold.
 \begin{enumerate}
\def\labelenumi{(\theenumi)} 
\item $\widehat{T}$ is a tilting $\widehat{\Lambda}$-module of projective dimension $m$. 
\item $\End_{\widehat{\Lambda}}(\widehat{T})$ is isomorphic to 
the (n+1)-preprojective algebra $\widehat{\Gamma}$ of $\Gamma$ as $\mathbb{Z}$-graded algebras. In particular, $\End_{\widehat{\Lambda}}(\widehat{T})$ is a graded bimodule (n+1)-CY algebra of Gorenstein parameter 1.
\end{enumerate}
 \end{thm}

First we show Theorem \ref{main_thm2} (1). 
We need the following observation.

\begin{lem}\label{tilde_lem2}
Let $0 \to X \to Y \to Z \to 0$ be an exact sequence in $\mod\Lambda$
such that $X,Y,Z \in \mathscr{N}_\Lambda^-$. 
Then, 
we have an exact sequence 
\[
0 \to \widehat{X} \to \widehat{Y} \to \widehat{Z} \to 0
\]
in $\gr \widehat{\Lambda}$.
\end{lem}

\begin{proof}
The exact sequence gives rise to a triangle $X \to Y \to Z \to X[1]$ in $\bdcat{}{\Lambda}$. 
By applying $\nu_n^{-i}$ to this triangle, we have a triangle $\nu_n^{-i}(X) \to \nu_n^{-i}(Y) \to \nu_n^{-i}(Z) \to \nu_n^{-i}(X[1])$ in $\bdcat{}{\Lambda}$. 
Then, by taking the homology, we have a long exact sequence
\begin{eqnarray*}
\cdots\cdots  \rightarrow \mathrm{H}^{j-1}(\nu_n^{-i}(Z))  \rightarrow \mathrm{H}^{j}(\nu_n^{-i}(X)) \rightarrow \mathrm{H}^{j}(\nu_n^{-i}(Y)) 
 \rightarrow \mathrm{H}^{j}(\nu_n^{-i}(Z)) \rightarrow \mathrm{H}^{j+1}(\nu_n^{-i}(X))  \rightarrow \cdots\cdots. 
\end{eqnarray*}
Because   $X,Y,Z \in \mathscr{N}_\Lambda^-$,
we have an exact sequence
\[
0 \to  \mathrm{H}^{0}(\nu_n^{-i}(X)) \rightarrow \mathrm{H}^{0}(\nu_n^{-i}(Y)) 
 \rightarrow \mathrm{H}^{0}(\nu_n^{-i}(Z))  \to 0.
\]
Since $\mathrm{H}^{0}(\nu_n^{-i}-) \simeq \tau^{-i}_n(-)$, 
we have the exact sequence $0 \to \tau^{-i}_n(X) \to \tau^{-i}_n(Y) \to \tau^{-i}_n(Z) \to 0$. 
Because $\tau^{-i}_n (X)=\widehat{X}_i$, 
we obtain the exact sequence $0 \to \widehat{X} \to \widehat{Y} \to \widehat{Z} \to 0$ in $\gr\widehat{\Lambda}$. 
\end{proof}

Next we show that the exact sequence \eqref{ast} induces a projective resolution of a simple $\widehat{\Lambda}$-module $S$,  
and $\widehat{K}_m$ is the $(n-m+1)$-th syzygy of $S$.
This fact plays an important role in the proof of Theorem \ref{main_thm2} (1).

\begin{prop}\label{tilde_prop1}
We regard $S$ as a $\mathbb{Z}$-graded simple $\widehat{\Lambda}$-module concentrated in degree $0$. 
Then there exists a projective resolution 
\begin{eqnarray}
0 \to \widehat{S} \xrightarrow{\widehat{a}_0} \widehat{P}_1 \xrightarrow{\widehat{a}_1} \cdots\cdots \xrightarrow{\widehat{a}_{n-1}}  \widehat{P}_n \xrightarrow{\widehat{a}_n} \widehat{S}(1) \to S(1) \to 0  \label{indu_mut}
\end{eqnarray}
of $S(1)$ in $\gr\widehat{\Lambda}$ such that $\widehat{P}_\ell \in \add\widehat{Q}$ and $\Im \widehat{a}_\ell\simeq\widehat{K}_\ell$ for any $1 \leq \ell\leq n$. 
\end{prop}

\begin{proof}
By applying Proposition \ref{van-lem1} and Lemma \ref{tilde_lem2} to the exact sequence \eqref{ast}, 
we have an exact sequence 
\[
0 \to \widehat{S} \xrightarrow{\widehat{a}_0} \widehat{P}_1 \xrightarrow{\widehat{a}_1} \cdots\cdots \xrightarrow{\widehat{a}_{n-1}}  \widehat{P}_n \xrightarrow{\widehat{a}_n} \widehat{\tau_n^-(S)} \to 0
\]
in $\gr \widehat{\Lambda}$, where $\Im \widehat{a}_i=\widehat{K}_i$. 
Moreover, there exists an exact sequence 
\[
0 \to \widehat{\tau_{n}^{-}(S)} \to \widehat{S}(1) \to S(1) \to 0
\]
in $\gr \widehat{\Lambda}$.
By combining these exact sequences, we have the required resolution.
\end{proof}

We remark that the above resolution can be understood as a graded version of the resolution   investigated in \cite[Section 4]{IR}.

Next we give a proof of Theorem \ref{main_thm2} (1).

\begin{proof}[Proof of Theorem \ref{main_thm2} (1)]
We prove that $\widehat{T}$ is a tilting $\widehat{\Lambda}$-module. 
By Proposition \ref{tilde_prop1}, we have an exact sequence
\begin{eqnarray}\label{exact}
0 \to  \widehat{S} \to \widehat{P}_1 \to \cdots\cdots \to  \widehat{P}_m \to \widehat{K}_m \to 0
\end{eqnarray}
in $\mod\widehat{\Lambda}$. Therefore,  $\widehat{T}$ satisfies (T1) and (T3).

Next we prove that $\widehat{T}$ satisfies (T2), that is,  $\Ext_{\widehat{\Lambda}}^i(\widehat{T},\widehat{T})=0$ for any $i\geq 1$. 
By Proposition \ref{tilde_prop1}, we have the following exact sequence
\[
0 \to \widehat{K}_m \to \widehat{P}_{m+1} \to \cdots\cdots  \to  \widehat{P}_n \to \widehat{S} \to S \to 0
\]
in $\mod\widehat{\Lambda}$. 
Thus we have $\Ext_{\widehat{\Lambda}}^i(\widehat{K}_m,\widehat{T}) \simeq \Ext_{\widehat{\Lambda}}^{n+i-m+1}(S,\widehat{T})$ for any $i\geq 1$.
Moreover, we have
\begin{eqnarray*}
\Ext_{\widehat{\Lambda}}^{n+i-m+1}(S,\widehat{T}) 
&\simeq& \Hom_{\mathcal{D}(\widehat{\Lambda})}(S,\widehat{T}[n+i-m+1]) \\
&\simeq& D\Hom_{\mathcal{D}(\widehat{\Lambda})}(\widehat{T}[n+i-m+1],S[n+1])\ \ \ \  (\mbox{Theorem } \ref{CY_condition})\\
&\simeq& D\Hom_{\mathcal{D}(\widehat{\Lambda})}(\widehat{T},S[m-i])
\end{eqnarray*}
for any $i\geq 1$.

Since $S$ is a simple projective $\La$-modules, we get $\Hom_{\widehat{\Lambda}}(\widehat{Q},S)\simeq\Hom_{\Lambda}({Q},S)=0$.  Then, using the sequence (\ref{exact}), 
we have $\Hom_{\mathcal{D}(\widehat{\Lambda})}(\widehat{T},S[m-i])=0$ because 
$\widehat{P}_\ell\in\add\widehat{Q}$ for $1\leq \ell \leq m$. 
Therefore we have $\Ext_{\widehat{\Lambda}}^i(\widehat{K}_m,\widehat{T})=0$ and hence $\Ext_{\widehat{\Lambda}}^i(\widehat{T},\widehat{T})=0$ for any $i\geq 1$. 
\end{proof}

Next we show the isomorphism of $\mathbb{Z}$-graded algebras in Theorem \ref{main_thm2} (2).
For this purpose, we introduce the following terminology.

\begin{df}
Let $X \in \mathscr{N}_\Lambda^-$. 
We consider an algebra
\[
\Upsilon(X):=\Upsilon_{\Lambda}(X)=\bigoplus_{i \geq 0} \Hom_{\bdcat{}{\Lambda}}(X,\nu_n^{-i}(X))
\]
whose multiplication is defined by
\[
 \Hom_{\bdcat{}{\Lambda}}(X,\nu_n^{-i}(X)) \times  \Hom_{\bdcat{}{\Lambda}}(X,\nu_n^{-j}(X)) \ni  (g,f) \mapsto \nu^{-j}_n(g)\circ f \in  \Hom_{\bdcat{}{\Lambda}}(X,\nu_n^{-i-j}(X)).
\]
This algebra can be regarded as a  $\mathbb{Z}$-graded algebra by $\Upsilon(X)_i=\Hom_{\bdcat{}{\Lambda}}(X,\nu_n^{-i}(X))$.
\end{df}

Then the following isomorphism holds (cf. \cite[Proposition 4.7]{Am}\cite[Proposition 2.12]{IO1}).

\begin{lem}\label{tilde_lem3}
Let $X \in \mathscr{N}_\Lambda^-$.
Then there exists an isomorphism $\End_{\widehat{\Lambda}}(\widehat{X}) \simeq \Upsilon_\La(X)$ of $\mathbb{Z}$-graded algebras.
\end{lem}
\begin{proof}
First we claim that the map 
\[
\rho:\Hom_{\widehat{\Lambda}}(\widehat{X},\widehat{X}(i))_0 \ni f \mapsto f|_{\widehat{X}_0} \in \Hom_{\Lambda}(X,\tau_n^{-i}(X))
\]
is an isomorphism  of groups for any $i \geq 0$.
We prove $\rho$ is monic. 
We take $f \in \Hom_{\widehat{\Lambda}}(\widehat{X},\widehat{X}(i))_0$ such that $\rho(f)=0$.
Then since $\widehat{X}$ is generated by the $0$-th part $\widehat{X}_0$, we have $f(\widehat{X})=f(\widehat{X}_0 \cdot \widehat{\Lambda})
=f(\widehat{X}_0)\widehat{\Lambda}=0$.
So we have $f=0$.

We prove $\rho$ is epic.
We take $g \in \Hom_{\Lambda}(X,\tau_n^{-i}(X))$.
Let $\iota:\widehat{\tau^{-i}_n(X)} \to \widehat{X}(i)$ be a natural embedding of $\mathbb{Z}$-graded $\widehat{\Lambda}$-modules.
We consider a composition $\iota  g \in \Hom_{\widehat{\Lambda}}(\widehat{X},\widehat{X}(i))_0$.
Then we have $\rho(\iota  g)=g$.
Thus the claim holds.

By the above claim and $X \in \mathscr{N}_\Lambda^-$, we have an isomorphism 
\[
\End_{\widehat{\Lambda}}(\widehat{X})
=\bigoplus_{i \geq 0} \Hom_{\widehat{\Lambda}}(\widehat{X},\widehat{X}(i))_0
\simeq  \bigoplus_{i \geq 0} \Hom_{\Lambda}(X,\tau_n^{-i}(X)) 
\simeq \bigoplus_{i \geq 0} \Hom_{\bdcat{}{\Lambda}}(X,\nu_n^{-i}(X)) \simeq \Upsilon_\La(X)
\]
of $\mathbb{Z}$-graded algebras.
\end{proof}

By Lemma \ref{Serre-com}, there exists an isomorphism 
\[
\varphi_{(-)} : \RHom_{\Lambda}(T,\nu_n^{-}(-)) \to \nu_n^{-}\RHom_{\Lambda}(T,-) 
\]
of functors.
For $i \geq 0$, we define a map $\varphi^{(i)}:\Hom_{\bdcat{}{\Lambda}}(T,\nu_n^{-i}T) \to \nu_n^{-i}\Hom_{\bdcat{}{\Lambda}}(T,T)$ by
\[
\varphi^{(i)}=
\begin{cases}
1 &  (i=0) \\
\nu_n^{-i+1}(\varphi_{T}) \circ \nu_n^{-i+2}(\varphi_{\nu_n^{-}T}) \circ \cdots \circ \nu_n^{-}(\varphi_{\nu_n^{-i+2}T}) \circ \varphi_{\nu_n^{-i+1}T}  & (i \geq 1).
\end{cases}
\]
Then one can check the following result.

\begin{lem} \label{tilde_lem4}
The following assertions hold.
\begin{enumerate}[(1)]
\item For $i \geq 0$, we have an isomorphism of groups $:$
\[
\Phi_i=\varphi^{(i)} \circ \Hom_{\Lambda}(T,-):
\Hom_{\bdcat{}{\Lambda}}(T,\nu_n^{-i}(T)) \to \Hom_{\bdcat{}{\Gamma}}(\Gamma,\nu_n^{-i}(\Gamma)) 
\]

\item We have an isomorphism of  $\mathbb{Z}$-graded algebras $:$
\[
\Phi=\bigoplus_{i \in \mathbb{Z}}\Phi_i:\Upsilon_{\Lambda}(T) \to \Upsilon_{\Gamma}(\Gamma).
\]

\end{enumerate}
\end{lem}

\begin{proof}
The first assertion is obvious. 
Moreover, one can check that the following diagram commutes :
\[
\xymatrix{
\Hom_{\bdcat{}{\Lambda}}(T,\nu_n^{-i}(T)) \times \Hom_{\bdcat{}{\Lambda}}(T,\nu_n^{-j}(T)) \ar[rr]^{\ \ \ \ \ \ \ \ \mbox{\small mult.}} \ar[d]_{\Phi_i \times \Phi_j}
& &\Hom_{\bdcat{}{\Lambda}}(T,\nu_n^{-i-j}(T)) \ar[d]^{\Phi_{i+j}}  \\
 \Hom_{\bdcat{}{\Gamma}}(\Gamma,\nu_n^{-i}(\Gamma)) \times  \Hom_{\bdcat{}{\Gamma}}(\Gamma,\nu_n^{-j}(\Gamma))  \ar[rr]_{\ \ \ \ \ \ \ \ \mbox{\small mult.}} 
 & & \Hom_{\bdcat{}{\Gamma}}(\Gamma,\nu_n^{-i-j}(\Gamma)). 
}
\]
 Thus, $\Phi$ is an isomorphism as $\mathbb{Z}$-graded algebras. 
\end{proof}

Now we are ready to finish the proof of Theorem \ref{main_thm2}.

\begin{proof}[Proof of Theorem \ref{main_thm2} (2)]
By Lemma \ref{tilde_lem3} and \ref{tilde_lem4}, we have isomorphisms
\begin{eqnarray*}
\End_{\widehat{\Lambda}}(\widehat{T}) 
\simeq \Upsilon_{\Lambda}(T)  
\simeq \Upsilon_{\Gamma}(\Gamma)  
\simeq \End_{\widehat{\Gamma}}(\widehat{\Gamma})
= \widehat{\Gamma}.
\end{eqnarray*}
of $\mathbb{Z}$-graded algebras. 
Then the second statement immediately follows from Theorem \ref{main_thm1}.
\end{proof}

As a corollary of Theorem \ref{main_thm2}, we show that $m$-APR tilting modules preserve $n$-representation tameness.  
Recall that an $n$-representation infinite algebra $\Lambda$ is called \emph{$n$-representation tame} if the center $\Cen(\widehat{\Lambda})$ of $\widehat{\Lambda}$ is noetherian and $\widehat{\Lambda}$ is finitely generated as a $\Cen(\widehat{\Lambda})$-module. 
In particular, $1$-representation tame algebras are hereditary representation tame algebras and hence $n$-representation tame algebras can be regarded as a generalization of them. 

\begin{cor}\label{pres_rep_tame}
$\Lambda$ is an $n$-representation tame if and only if so is $\Gamma$.
\end{cor}

\begin{proof}
By Theorem \ref{main_thm2} (2), $\widehat{\La}$ and $\widehat{\Gamma}$ are derived equivalent. 
Then the assertion follows from \cite[Proposition 9.4]{Ric}.
\end{proof}

For the case of  $m=n$, $m$-APR tilting modules have a particularly nice property as stated below.
It  can be regarded as a generalization of  \cite[Proposition II.1.4.]{BIRSc}, and will be proved by the same argument and Proposition \ref{tilde_prop1}.

\begin{prop}\label{n-pre-iso} 
Assume that $T$ is an $n$-APR tilting $\Lambda$-module. Then there exists an isomorphism 
$\widehat{\Lambda}\simeq \widehat{\Gamma}$ of algebras.
\end{prop}

\begin{proof} 
By $T=Q\oplus \tau^-_n(S)$ and Proposition \ref{tilde_prop1}, we have an exact sequence
\begin{eqnarray}
0 \to \widehat{T} \xrightarrow{f} \widehat{\Lambda} \to S \to 0 \label{ex_ideal}
\end{eqnarray}
in $\mod\widehat{\Lambda}$.

First, by applying the functor $\Hom_{\widehat{\Lambda}}(\widehat{T},-)$ to the exact sequence \eqref{ex_ideal}, we have the following exact sequence 
\[
0 \to \Hom_{\widehat{\Lambda}}(\widehat{T},\widehat{T}) \xrightarrow{\Hom_{\widehat{\Lambda}}(\widehat{T},f)}
 \Hom_{\widehat{\Lambda}}(\widehat{T},\widehat{\Lambda}) \xrightarrow{\hspace{8mm}} 
 \Hom_{\widehat{\Lambda}}(\widehat{T},S).
\]
By $\Hom_{\widehat{\Lambda}}(\widehat{T},S)=0$, the map $\Hom_{\widehat{\Lambda}}(\widehat{T},f)$ is an isomorphism.

Next, by applying the functor $\Hom_{\widehat{\Lambda}}(-,\widehat{\Lambda})$ to the exact sequence \eqref{ex_ideal}, we have the following exact sequence 
\[
\Hom_{\widehat{\Lambda}}(S,\widehat{\Lambda}) \xrightarrow{\hspace{8mm}} 
\Hom_{\widehat{\Lambda}}(\widehat{\Lambda},\widehat{\Lambda}) \xrightarrow{\Hom_{\widehat{\Lambda}}(f,\widehat{\Lambda})} 
\Hom_{\widehat{\Lambda}}(\widehat{T},\widehat{\Lambda}) \xrightarrow{\hspace{8mm}} 
\Ext^1_{\widehat{\Lambda}}(S,\widehat{\Lambda}).
\]
Then, by Theorem \ref{CY_condition}, we have $\Hom_{\mathcal{D}(\widehat{\Lambda})}(S,\widehat{\Lambda}[i]) \simeq D\Hom_{\mathcal{D}(\widehat{\Lambda})}(\widehat{\Lambda},S[n-i+1])=0$ for $i=0,1$. 
Thus the map $\Hom_{\widehat{\Lambda}}(f,\widehat{\Lambda})$ is an isomorphism. 
Therefore we have an isomorphism of groups  
\[
\Hom_{\widehat{\Lambda}}(\widehat{T},f)^{-1} \circ \Hom_{\widehat{\Lambda}}(f,\widehat{\Lambda}) :
\Hom_{\widehat{\Lambda}}(\widehat{\Lambda},\widehat{\Lambda}) \to \Hom_{\widehat{\Lambda}}(\widehat{T},\widehat{T}).
\]

One can easily check that this gives an isomorphism of algebras. 
\end{proof}

At the end of this paper, we give some examples of Theorem \ref{main_thm1}, \ref{main_thm2} and Proposition \ref{n-pre-iso} in the case of $n\leq 2$. In this case, we can explicitly describe quivers of relations of the algebras as below. 

\begin{ex}\label{example} \
\begin{itemize}
\item[(1)]
First we give an example for the classical case, namely the case of $n=m=1$.
Let $Q$ be the following quiver.
$$
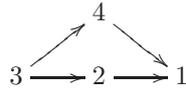
\xymatrix@C20pt@R12pt{
&4\ar[dr] & \\
3 \ar[r]\ar[ur]  &\ar[r]  2& 1}$$
We consider the path algebra $\La:=KQ$ of $Q$, which is $1$-representation infinite (cf. \cite[VII. Theorem 5.10]{ASS}), and the $1$-APR tilting $\Lambda$-module $T$ associated with vertex 1.  
Then $\Gamma:=\End_\La(T)$  is also a $1$-representation infinite algebra, which is the path algebra of the quiver obtained from $Q$ by reversing the arrows ending at the vertex 1 \cite{BGP,APR}.

It is known that the 2-preprojetive algebras $\widehat{\Lambda}$ and $\widehat{\Gamma}$ are given by the double quiver of the quiver of $\Lambda$ and $\Gamma$ with some relations respectively (see \cite{Rin}). 
Moreover $T$ induces a tilting $\widehat{\Lambda}$-module $\widehat{T}$ with $\widehat{\Gamma} \simeq \End_{\widehat{\Lambda}}(\widehat{T}) \simeq \widehat{\Lambda}$ by Theorem \ref{main_thm2} and Proposition \ref{n-pre-iso}.

These results imply the compatibility of the following diagram of quivers, 
where horizontal arrows indicate tilts of $T$ and $\widehat{T}$, respectively, and vertical arrows indicate taking $2$-preprojective algebras.

$$
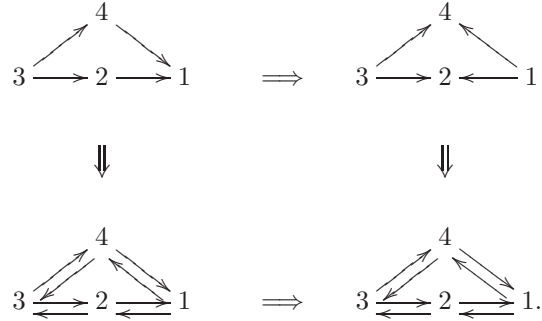
\xymatrix@C20pt@R12pt{
&4\ar[dr] & &\\
3 \ar[r]\ar[ur]  &\ar[r]  2& 1&{\Longrightarrow}\\
&\ar@{=>}[d]&&\\ 
&&&\\ 
   &4\ar[dr] \ar@<1ex>[dl]& &\\
3 \ar[r]\ar[ur]  &\ar[r]\ar@<1ex>[l]  2& \ar@<1ex>[l] \ar@<1ex>[ul] 1& {\Longrightarrow}
}\ \ \ \ 
\xymatrix@C20pt@R12pt{
    &4 & &\\
3  \ar[r] \ar[ur]& 2& 1\ar[l] \ar[ul] &\\
&\ar@{=>}[d]&&\\  
&&&\\ 
  &4\ar[dr] \ar@<1ex>[dl] & &\\
3 \ar[r]\ar[ur]  &\ar[r] \ar@<1ex>[l]  2& \ar@<1ex>[l] \ar@<1ex>[ul] 1.& \\
}$$

\bigskip

\item[(2)] 
Next we give an example for the case $m=1<2=n$.
We note that the structure of $3$-CY algebras has been extensively studied (for example \cite{AIR,Bro,Boc,G,TV})  
and it is known that they have a close relationship with quivers with potentials (QPs) in the sense of  \cite{DWZ}. 

Let $Q$ be a quiver 
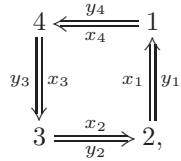
\[
\xymatrix@C30pt@R30pt{
4\ar@2{->}[d]^{x_3}_{y_3}    &1\ar@2{->}[l]^{x_4}_{y_4}  \\
3 \ar@2{->}[r]^{x_2}_{y_2}  &  2,\ar@2{->}[u]^{x_1}_{y_1}    }
\]
and $W:=x_1x_2x_3x_4-y_1y_2y_3y_4+x_1y_2x_3y_4-y_1x_2y_3x_4$  a \emph{potential} on $Q$ and $C:=\{x_4,y_4\}$ a \emph{cut} (see \cite{HI2} for the terminology). 
Then the \emph{truncated Jacobian algebra} $\Lambda$ of $(Q,W,C)$ is a $2$-representation infinite algebra (see \cite[section 6]{AIR}), whose quiver is the left upper one in the picture below. 
We can consider the $1$-APR tilting $\Lambda$-module $T$ associated with vertex 1. 
By Theorem \ref{main_thm1}, $\Gamma:=\End_\La(T)$ is also a $2$-representation infinite algebra.
Moreover $T$ induces a tilting $\widehat{\Lambda}$-module $\widehat{T}$ with $\widehat{\Gamma} \simeq \End_{\widehat{\Lambda}}(\widehat{T})$.

In this example, we can understand the change of quivers with relations of tilts and the $3$-preprojective algebras. 
Indeed, it is known that the quiver with relations of $\Gamma$ can be calculated by applying mutation of graded QPs \cite{M}. 
On the other hand, the $3$-preprojective algebra $\widehat{\Lambda}$ is given as the \emph{Jacobian algebra} of $(Q,W)$ (see \cite{Ke2}), and $\End_{\widehat{\Lambda}}(\widehat{T})$ is given as the Jacobian algebra of the QP obtained by mutating $(Q,W)$ \cite{BIRSm,KY}.

Therefore, we have the following diagram of quivers, 
where horizontal arrows indicate tilts of $T$ and $\widehat{T}$, respectively, and vertical arrows indicate taking $3$-preprojective algebras.

$$
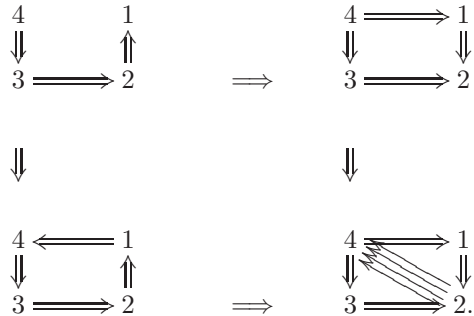
\xymatrix@C30pt@R12pt{
4\ar@2{->}[d]   &1 &\\
3 \ar@2{->}[r]  &  2\ar@2{->}[u]&{\Longrightarrow}\\
\ar@{=>}[d]_{}&&\\
\\
4\ar@2{->}[d]    &1\ar@2{->}[l] & \\
3 \ar@2{->}[r]  &  2\ar@2{->}[u] & {\Longrightarrow} 
}\ \ \ \ \ \ 
\xymatrix@C30pt@R12pt{
 4 \ar@2{->}[d] \ar@2{->}[r] &1 \ar@2{->}[d]  \\
3  \ar@2{->}[r]  &  2 \\
\ar@{=>}[d]_{}&&\\
\\
 4 \ar@2{->}[d] \ar@2{->}[r] &1  \ar@2{->}[d]  \\
3  \ar@2{->}[r]  &  2.\ar@<1.5mm>[lu]\ar@<.5mm>[lu]\ar@<-.5mm>[lu]\ar@<-1.5mm>[lu]  }$$

Note that the algebra $\widehat{\La}$ arises from as a rolled-up helix on a del Pezzo surface \cite{BS} and, from a result of \cite{TV}, the QP is \emph{non-degenerate}. 
Thus, we can apply mutation of QPs repeatedly and obtain a large family of these algebras.

\item[(3)]
Finally we give an example for the case $n=m=2$.
Let $Q$ be a quiver
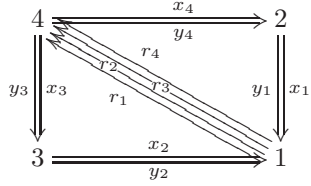
\[\xymatrix@C80pt@R40pt{
4 \ar@2{->}[d]^{x_3}_{y_3} \ar@2{->}[r]^(.6){x_4}_(.6){y_4} &2  \ar@2{->}[d]^{x_1}_{y_1}  \\
3  \ar@2{->}[r]^{x_2}_{y_2}  &  1\ar@<1.5mm>[lu]^(.6){r_1}\ar@<.5mm>[lu]|(.7){r_2}\ar@<-.5mm>[lu]|(.5){r_3}\ar@<-1.5mm>[lu]_(.6){r_4} 
}\]
and $W:=(x_1x_4-x_2x_3)r_1+(x_1y_4-y_2x_3)r_2+(y_1x_4-x_2y_3)r_3+(y_1y_4-y_2y_3)r_4$ a potential on $Q$ and 
$C:=\{r_1,r_2,r_3,r_4\}$ a cut.
Then the truncated Jacobian algebra $\Lambda$ of $(Q,W,C)$ is a $2$-representation infinite algebra given in \cite[Example 2.14]{HIO}, whose quiver is the left upper one in the picture below. 
We can consider the $2$-APR tilting $\Lambda$-module $T$ associated with vertex $2$.
By Theorem \ref{main_thm1}, $\Gamma:=\End_\La(T)$ is  a $2$-representation infinite algebra  (this also follows from \cite[Theorem 2.13]{HIO}).
Moreover $T$ induces a tilting $\widehat{\Lambda}$-module $\widehat{T}$ with $\widehat{\Gamma} \simeq \End_{\widehat{\Lambda}}(\widehat{T}) \simeq \widehat{\Lambda}$.

In this example, 
the quiver of $\Gamma$ can be calculated by the same argument of \cite[Theorem 3.11]{IO1}, 
and the $3$-preprojective alegbra $\widehat{\Lambda}$ is given as the Jacobian algebra of $(Q,W)$.

Thus, we have the following diagram of quivers, 
where horizontal arrows indicate tilts of $T$ and $\widehat{T}$, respectively, and vertical arrows indicate taking $3$-preprojective algebras.

$$
\xymatrix@C30pt@R12pt{
4 \ar@2{->}[d] \ar@2{->}[r] &1  \ar@2{->}[d]& \\
3  \ar@2{->}[r]  &  2  &{\Longrightarrow}   \\ 
\ar@{=>}[d]_{}&&\\
 &\\ 
4 \ar@2{->}[d] \ar@2{->}[r] &1  \ar@2{->}[d] & \\
3  \ar@2{->}[r]  &  2 \ar@<1.5mm>[lu]^(.6){}\ar@<.5mm>[lu]|(.7){}\ar@<-.5mm>[lu]|(.5){}\ar@<-1.5mm>[lu]_(.6){} &  {\Longrightarrow}  }
\ \ \ \ \ \xymatrix@C30pt@R12pt{
4 \ar@2{->}[d] \ar@2{->}[r] &1    \\
3    &  2 \ar@<1.5mm>[lu]^(.6){}\ar@<.5mm>[lu]|(.7){}\ar@<-.5mm>[lu]|(.5){}\ar@<-1.5mm>[lu]_(.6){} &  \\
 \ar@{=>}[d]_{}&\\
&\\ 
 4 \ar@2{->}[d] \ar@2{->}[r] &1  \ar@2{->}[d]  \\
3  \ar@2{->}[r] &  2 \ar@<1.5mm>[lu]^(.6){}\ar@<.5mm>[lu]|(.7){}\ar@<-.5mm>[lu]|(.5){}\ar@<-1.5mm>[lu]_(.6){}   }$$
\end{itemize}
\end{ex}

\end{document}